\documentclass[11pt,a4paper,leqno]{amsart}
  \usepackage{silence}
\usepackage[english]{babel}
\usepackage[dvipsnames]{xcolor}
\usepackage{amsthm,graphicx,pifont}
\usepackage[utopia]{mathdesign}
\usepackage[a4paper,top=3cm,bottom=3cm,left=3cm,right=3cm,marginparwidth=60pt]{geometry}
\usepackage[utf8]{inputenc}
\usepackage{thmtools,braket,caption,comment,mathtools,stmaryrd,enumitem,amsmath} 
\usepackage[usestackEOL]{stackengine}
\usepackage{multirow,booktabs,microtype,relsize}
\usepackage[foot]{amsaddr}
\usepackage[colorlinks,bookmarks]{hyperref} %
      \hypersetup{colorlinks,%
            citecolor=britishracinggreen,%
            filecolor=black,%
            linkcolor=cobalt,%
            urlcolor=purple}
      \numberwithin{equation}{section}
\usepackage[capitalise]{cleveref} 
\allowdisplaybreaks 
\usepackage{stmaryrd}

\DeclareSymbolFont{cmarrows}{OMS}{cmsy}{m}{n}
\SetSymbolFont{cmarrows}{bold}{OMS}{cmsy}{b}{n}
\DeclareMathSymbol{\cmminus}{\mathbin}{cmarrows}{"00}

\DeclareMathSymbol{\leftrightarrow}{\mathrel}{cmarrows}{"24}
\DeclareMathSymbol{\leftarrow}{\mathrel}{cmarrows}{"20}
   
\DeclareMathSymbol{\rightarrow}{\mathrel}{cmarrows}{"21}
   \let\to=\rightarrow
\DeclareMathSymbol{\mapstochar}{\mathrel}{cmarrows}{"37}
   \def\mapsto{\mapstochar\rightarrow}
   
\DeclareSymbolFont{usualmathcal}{OMS}{cmsy}{m}{n}
\DeclareSymbolFontAlphabet{\mathcal}{usualmathcal}
\DeclareMathAlphabet\BCal{OMS}{cmsy}{b}{n}

\makeatletter
\newcommand{\mylabel}[2]{#2\def\@currentlabel{#2}\label{#1}}
\makeatother

\renewcommand{\leq}{\leqslant}
\renewcommand{\geq}{\geqslant}

\definecolor{cornellred}{rgb}{0.7, 0.11, 0.11}
\definecolor{britishracinggreen}{rgb}{0.0, 0.26, 0.15}
\definecolor{cobalt}{rgb}{0.0, 0.28, 0.67}

\newcommand*{\defeq}{\mathrel{\vcenter{\baselineskip0.5ex \lineskiplimit0pt
                     \hbox{\scriptsize.}\hbox{\scriptsize.}}}%
                     =}

\usepackage[colorinlistoftodos]{todonotes} 
\usepackage{amsmath,float,soul}
\usetikzlibrary{decorations.markings}

\newcommand{\Flag}{\mathrm{Fl}}
\newcommand{\Gr}{\mathrm{Gr}}

\newcommand{\into}{\hookrightarrow}
\newcommand{\onto}{\twoheadrightarrow}
\newcommand{\HH}{\mathrm{H}}
\newcommand{\OO}{\mathscr O}

\newcommand{\boldit}[1]{\boldsymbol{#1}}

\DeclareMathOperator{\rk}{rk}
\DeclareMathOperator{\lHom}{\mathscr{H}\kern-0.3em\mathit{om}}
\DeclareMathOperator{\LQuot}{LQuot}
\DeclareMathOperator{\LHilb}{LHilb}
\DeclareMathOperator{\Hilb}{Hilb}

\DeclareMathOperator{\mdeg}{mdeg}

\DeclareMathOperator{\Sym}{Sym}
\DeclareMathOperator{\Var}{Var}

\DeclareMathOperator{\In}{In}

\DeclareMathOperator{\Hom}{Hom}

\DeclareMathOperator{\Quot}{Quot}

\DeclareMathOperator{\Spec}{Spec}

\DeclareMathOperator{\Supp}{Supp}

\newcommand{\BA}{{\mathbb{A}}}

\newcommand{\BC}{{\mathbb{C}}}

\newcommand{\BG}{{\mathbb{G}}}

\newcommand{\BL}{{\mathbb{L}}}

\newcommand{\BN}{{\mathbb{N}}}

\newcommand{\BP}{{\mathbb{P}}}
\newcommand{\BQ}{{\mathbb{Q}}}

\newcommand{\BZ}{{\mathbb{Z}}}

\newcommand{\CE}{{\mathcal E}}
\newcommand{\CF}{{\mathcal F}}

\newcommand{\CH}{{\mathcal H}}

\newcommand{\CJ}{{\mathcal J}}
\newcommand{\CK}{{\mathcal K}}

\newcommand{\CM}{{\mathcal M}}

\newcommand{\CQ}{{\mathcal Q}}

\newcommand{\CT}{{\mathcal T}}

\newcommand{\Fm}{{\mathfrak{m}}}

\usepackage[all]{xy}
\usepackage{tikz}
\usepackage{tikz-cd}
\usepackage{adjustbox}
\usepackage{rotating}
\usepackage{comment}

\usetikzlibrary{matrix,shapes,intersections,arrows,decorations.pathmorphing}
\tikzset{commutative diagrams/.cd,
mysymbol/.style={start anchor=center,end anchor=center,draw=none}}

\tikzset{
shift up/.style={
to path={([yshift=#1]\tikztostart.east) -- ([yshift=#1]\tikztotarget.west) \tikztonodes}
}
}
  
\theoremstyle{definition}

\newtheorem*{lemma*}{Lemma}
\newtheorem*{theorem*}{Theorem}
\newtheorem*{example*}{Example}
\newtheorem*{fact*}{Fact}
\newtheorem*{notation*}{Notation}
\newtheorem*{definition*}{Definition}
\newtheorem*{prop*}{Proposition}
\newtheorem*{remark*}{Remark}
\newtheorem*{corollary*}{Corollary}

\newtheorem*{conventions*}{Conventions}

\newtheorem{definition}{Definition}[section]

\newtheorem{remark}[definition]{Remark}

\newtheoremstyle{thm} 
        {3mm}
        {3mm}
        {\slshape}
        {0mm}
        {\bfseries}
        {.}
        {1mm}
        {}
\theoremstyle{thm}
\newtheorem{theorem}[definition]{Theorem}
\newtheorem{corollary}[definition]{Corollary}
\newtheorem{lemma}[definition]{Lemma}
\newtheorem{prop}[definition]{Proposition}

\newtheorem{thm}{Theorem}

\newtheorem{cor}[thm]{Corollary}

\newtheorem*{Acknowledgments*}{Acknowledgments}

\makeatletter
\newenvironment{proofof}[1]{\par
  \pushQED{\qed}%
  \normalfont \topsep6\p@\@plus6\p@\relax
  \trivlist
  \item[\hskip3\labelsep
        \itshape
    Proof of #1\@addpunct{.}]\ignorespaces
}{%
  \popQED\endtrivlist\@endpefalse
}
\makeatother

\usepackage[baseline]{euflag}

\title[Motivic and cohomological stabilisation of the Quot scheme of points]{Motivic and cohomological stabilisation \\ of the Quot scheme of points}
\author{Alexis Aumonier, Michele Graffeo, Sergej Monavari, Riccardo Moschetti, Andrea T. Ricolfi}
\keywords{Quot schemes, Motives, Grothendieck ring of varieties, cohomology}
\subjclass[2020]{Primary 14C05; Secondary 14C15.}

\begin{document}
\begin{abstract}
We prove that the motive of the punctual Quot scheme $\Quot^d(\OO^{\oplus r}_{\BA^n})_0$ stabilises, when $n \to \infty$, to $[\Gr(d-1,\infty)]\cdot \sum_{i=0}^{r-1}\BL^{di}$. We similarly prove that the stable cohomology ring $\HH^\ast(\Quot^d(\OO^{\oplus r}_{\BA^\infty}), \BZ)$ is graded-isomorphic to the ring $\BZ[c_1,\ldots,c_{d}]/(c_d^r)$, where $\deg c_i = 2i$, answering in the positive a question of Pandharipande. 
Finally, we prove that the motive of the nested Hilbert scheme stabilises to the motive of the infinite flag variety and we compute the cohomology ring in the limit. 
\end{abstract}

\maketitle
{\hypersetup{linkcolor=black}\tableofcontents}

\section{Introduction}
One of the most classical and well-understood moduli spaces in algebraic geometry is the Grassmannian $\Gr(d,n)$, whose study was initiated by Schubert in the 19th century. By Giambelli's theorem, its cohomology ring is generated by the Chern classes of the tautological subbundle $\mathscr S$. All the relations are encoded in the identity $c(\mathscr S)c(\mathscr Q) - 1= 0$, where $\mathscr Q$ is the tautological quotient. Because the relations in $\Gr(d,n)$ start in degree $n-d+1$, they go away when $n \to \infty$, and this yields a ring isomorphism 
\begin{equation}
\label{Grass-stab}
\HH^\ast (\Gr(d,\infty),\BZ) = \varprojlim \HH^\ast (\Gr(d,n),\BZ) \cong \BZ[c_1,\ldots,c_d],
\end{equation}
with $c_i$ of degree $2i$ for $i=1,\ldots,d$. About a century later, Mumford \cite{Mumford1983} initiated the enumerative geometry of the moduli space $\CM_g$ of curves of genus $g$ by making an explicit comparison with the Grassmannian: he introduced the kappa classes $\kappa_i \in \HH^{2i}(\CM_g,\BQ)$ as analogues of the Chern classes of the tautological subbundle on $\Gr(d,n)$. He defined the \emph{tautological ring} $\mathrm{R}^\ast (\CM_g,\BQ) \subset \HH^\ast(\CM_g,\BQ)$ as the subring generated by the kappa classes. Their relations determine the ring structure of $\mathrm{R}^\ast (\CM_g,\BQ)$. Once again, they go away when $g \to \infty$, in the sense that
\begin{equation}
\label{Mg-stab}
\varprojlim \mathrm{R}^\ast (\CM_g,\BQ) = \varprojlim \HH^\ast (\CM_g,\BQ) \cong \BZ[\kappa_1,\kappa_2,\kappa_3,\ldots].
\end{equation}
This was Mumford's conjecture, proved by Madsen--Weiss in \cite{Madsen-Weiss}.
The behaviour described in \eqref{Grass-stab} and \eqref{Mg-stab} is known as \emph{stabilisation of cohomology}. This paper studies precisely this phenomenon in the case where the finite-level moduli space of interest is the \emph{Quot scheme of points}
\begin{equation}
\label{quot}
\Quot^d(\OO_{\BA^n}^{\oplus r}),
\end{equation}
whose motivic theory has recently attracted a lot of attention \cite{MR_nested_Quot,mozgovoy2019motivic,MOTIVEMOTIVEMOTIVE,FGLR,monavari2024hyperquot,double-nested-1} (see also \cite{DavisonR,cazzaniga2020higher,Virtual_Quot} for computations of \emph{virtual} motivic invariants of Quot schemes).
This time, the moduli space \eqref{quot} is highly singular and not well-understood from a geometric point of view, let alone cohomological. 

Our main result is the following.

\begin{thm}[\Cref{main-thm-pandha}]
\label{panda-conjecture}
Fix integers $d,r\geq 1$. Then, there is a graded ring isomorphism
\[
\begin{tikzcd}
\BZ[c_1,\ldots,c_{d}]/(c_d^r) \arrow{r}{\sim} & \HH^\ast(\Quot^d(\OO^{\oplus r}_{\BA^\infty}), \BZ),
\end{tikzcd}
\]
where $\deg c_i=2i$ for each $i=1$.
\end{thm}

According to the recent paper \cite{Pielasa}, Pandharipande originally raised the question of whether such an isomorphism exists, which we answer in the positive.

Here, $\Quot^d(\OO^{\oplus r}_{\BA^\infty})$ is the colimit $\varinjlim \Quot^d(\OO^{\oplus r}_{\BA^n})$ of the spaces \eqref{quot}, viewed as an ind-scheme. \Cref{panda-conjecture} is already established in the case $r=1$ (the Hilbert scheme case), as a consequence of the work \cite{Hilb^infinity}, which exploits $\BA^1$-homotopy techniques. This also follows from \Cref{thm:flag-intro}\ref{coh}. In the paper \cite{Pielasa} it is shown that the two algebras involved in \Cref{panda-conjecture} have the same Poincar\'e series for $d\leq 2$. In this paper we extend this result to arbitrary $d$, as a consequence of \Cref{thm:main-B}. This is the key step in the proof of \Cref{panda-conjecture}.

\smallbreak
Our starting point is an analysis of the stabilisation of cohomology via motivic techniques, partially building upon the methods introduced by the last four authors in \cite{MOTIVEMOTIVEMOTIVE}. Note that the ring $\BZ[c_1,\ldots,c_{d}]/(c_d^r)$ can be written as $\HH^\ast (\Gr(d-1,\infty),\BZ) \otimes_{\BZ}\BZ[c_d] / c_d^r$, with $c_d$ (somewhat crucially!) in degree $2d$. The next result is the motivic counterpart of this decomposition.

\begin{thm}[\Cref{main-A}]
\label{thm:main-intro-A}
Fix integers $d,r\geq 1$. There is an identity
    \begin{align*}
        [\Quot^d(\OO^{\oplus r}_{\BA^\infty})_0]=[\Gr(d-1, \infty)]\cdot \sum_{i=0}^{r-1} \BL^{di}\in \widehat{K_0}(\Var_\BC).
    \end{align*}
\end{thm}

In the statement, $\widehat{K_0}(\Var_\BC)$ denotes the $\BL$-adic completion of the Grothendieck ring of varieties $K_0(\Var_\BC)$, and $\Quot^d(\OO^{\oplus r}_{\BA^\infty})_0 \subset \Quot^d(\OO^{\oplus r}_{\BA^\infty})$ is the ind-scheme defined as the colimit $\varinjlim \Quot^d(\OO^{\oplus r}_{\BA^n})_0$, where  $\Quot^d(\OO^{\oplus r}_{\BA^n})_0 \subset \Quot^d(\OO^{\oplus r}_{\BA^n})$ is the \emph{punctual Quot scheme}. Note that, as we prove in \Cref{lemma:infinite-grass-motive}, the motive of the ind-Grassmannian $\Gr(d-1,\infty) = \varinjlim \Gr(d-1,n)$ admits the explicit expression
\[
[\Gr(d-1, \infty)] = \prod_{k=1}^{d-1}\frac{1}{1-\BL^k},
\]
therefore the formula in \Cref{thm:main-intro-A} is fully explicit.

A key character for our argument is the \emph{linear Quot scheme}, namely the closed subset
\[
\LQuot^d(\OO_{\BA^n}^{\oplus r}) \subset \Quot^d(\OO_{\BA^n}^{\oplus r})_0
\]
parametrising quotients $[\OO_{\BA^n}^{\oplus r} \onto T]$ supported at the origin such that the Hilbert--Samuel function of $T$ has length $0$ or $1$. The main observation is that \emph{only} the Hilbert--Samuel strata building up $\LQuot^d(\OO_{\BA^n}^{\oplus r})$ `survive' motivically in the limit (cf.~\Cref{prop: non linear pieces}). Moreover, they are smooth quasiprojective varieties of a very explicit form: in \Cref{prop: linear locus} we show that they are vector bundles over certain (explicit) Grassmann bundles over a Grassmannian, and this makes their motive easily accessible. Summing up the contributions of such linear strata, we obtain in \Cref{eqn: motive linear quot} the key identity
\[
[\LQuot^d(\OO_{\BA^n}^{\oplus r})]=\sum_{i=1}^d\,[\Gr(i,r)][\Gr(d-i,in)]\BL^{(d-i)(r-i)},
\]
whose limit for $n \to \infty$ gives, after some work, the formula in \Cref{thm:main-intro-A}. 
 It is remarkable that, despite  having a very explicit formula for the motive of $\LQuot^d(\OO_{\BA^n}^{\oplus r})$, its geometry shares many pathologies with the Quot scheme $ \Quot^d(\OO_{\BA^n}^{\oplus r})$. For instance, already for $n=d=4$ and $r=2$, one can check that the elementary component of $\Quot^4(\OO_{\BA^4}^{\oplus 2})$, see \cite{Joachim-Sivic}, can also be found as an irreducible component of $\LQuot^4(\OO_{\BA^4}^{\oplus 2})$.  This component corresponds to the Hilbert--Samuel function $(2,2)$.

A crucial step in the proof of \Cref{panda-conjecture} is the following result on the structure of the cohomology of the infinite Quot scheme.

\begin{thm}[\Cref{thm:purity}, \Cref{cor:poincare}]
\label{thm:main-B}
Fix integers $d,r,n \geq 1$. 
\begin{itemize}
\item [\mylabel{purity}{\normalfont(1)}] Deligne's mixed Hodge structure on the cohomology $\HH^k(\LQuot^d(\OO_{\BA^n}^{\oplus r}), \BQ)$ is pure of Tate type for all $k$. The integral cohomology $\HH^k(\LQuot^d(\OO_{\BA^n}^{\oplus r}), \BZ)$ is torsion-free. In particular, there is an identity
\[
P(\LQuot^d(\OO_{\BA^n}^{\oplus r}), z)=\sum_{i=1}^dP(\Gr(i,r), z)P(\Gr(d-i,in), z)z^{2(d-i)(r-i)},
\]
where $P(Y,z)$ denotes the Poincar\'e series of a space $Y$.
\item [\mylabel{poincare}{\normalfont(2)}]   The restriction  
\[
\begin{tikzcd}
    \HH^*\left(\Quot^d(\OO^{\oplus r}_{\BA^n}), \BZ\right)\arrow[r] &\HH^*\left(\LQuot^d(\OO^{\oplus r}_{\BA^n}), \BZ\right)
\end{tikzcd}
\]
is an isomorphism in degrees at most $2\left( n+r-d\right)$. In particular, 
there is a graded isomorphism of cohomology rings 
\[
\HH^*\left(\Quot^d(\OO^{\oplus r}_{\BA^\infty}), \BZ\right)\cong \HH^*\left(\LQuot^d(\OO^{\oplus r}_{\BA^\infty}), \BZ\right),
\]
and an identity
\[
P(\Quot^d(\OO^{\oplus r}_{\BA^\infty}), z) = \left(\prod_{k=1}^{d-1}\frac{1}{1-z^{2k}}\right)\cdot \sum_{i=0}^{r-1}z^{2di}.
\]
\end{itemize}
\end{thm}

\Cref{thm:main-B} implies that the algebras in \Cref{panda-conjecture} are free  abelian groups and have the same Poincar\'e series. From here, the proof of \Cref{panda-conjecture} we provide in \Cref{sec:main-proof} proceeds as follows. We first use a tautological vector bundle $\CE$ on the Quot scheme $\Quot^d(\OO^{\oplus r}_{\BA^\infty})$ to define a morphism
\[
\begin{tikzcd}
    \BZ[c_1,\ldots,c_{d}]/(c_d^r)\arrow[r]& \HH^\ast(\Quot^d(\OO^{\oplus r}_{\BA^\infty}), \BZ),
\end{tikzcd}
\]
sending $c_i$ to the $i$th Chern class of $\CE$. We then show that this morphism is injective modulo any prime $p$, using an auxiliary morphism and Chern class computations. We finally invoke \cref{thm:main-B} to conclude that the original morphism is degree-wise a morphism of determinant $\pm 1$ between free abelian groups of the same finite rank, hence an isomorphism.
\smallbreak
Moreover, we  give  a topological interpretation of the stable cohomology appearing in \Cref{panda-conjecture}.
\begin{cor}[\Cref{cor: sphere}]\label{cor: sphere INTRO}
Fix an integer $r \geq 1$. Let $\gamma_d^{\oplus r} \setminus \{0\}$ denote the complement of the zero section of the $r$-fold sum of the tautological rank $d$ vector bundle $\gamma_d$ on the infinite Grassmannian $\Gr(d,\infty)$. Then there is a continuous map (in the analytic topology)
\[
\begin{tikzcd}
\Quot^d(\OO^{\oplus r}_{\BA^\infty}) \arrow{r} &  \gamma_d^{\oplus r} \setminus \{0\}
\end{tikzcd}
\]
which induces the isomorphism of \cref{panda-conjecture} on integral cohomology.
\end{cor}

Finally, in \Cref{sec:nested-hilb} we consider the (infinite) nested Hilbert scheme and show it stabilises, motivically and cohomologically, to the (infinite) flag variety.

\begin{thm}[\Cref{thm:nested},\,\Cref{cor:nested-cohomology}]
\label{thm:flag-intro}
Fix integers $\ell\geq 1$ and $0\leq d_1\leq \cdots \leq d_\ell \leq n$. 
\begin{itemize}
\item [\mylabel{stab}{\normalfont{(1)}}] There is an identity
\[
[\Hilb^{d_1+1,\ldots,d_\ell+1}(\BA^n)_0] =[\Flag(d_1,\ldots, d_\ell,n)] \in K_0(\Var_\BC)/\BL^{n-d_\ell+1}K_0(\Var_\BC).
\]
In particular, the  sequence of motives $([\Hilb^{d_1+1,\ldots,d_\ell+1}(\BA^n)_0])_n$ converges in $ \widehat{K}_0(\Var_\BC)$ and in $\widehat{K}_0(\Var_\BC)\llbracket t \rrbracket$ we have the identity
\begin{align*}
\sum_{0\leq {d_1\leq\cdots\leq d_\ell}}\,[\Hilb^{d_1+1,\ldots,d_\ell+1}(\BA^\infty)_0]t_1^{d_1}\cdots t_\ell^{d_\ell} &=\prod_{j=0}^{\ell-1}\prod_{i\geq 0}\frac{1}{1-\BL^it_{j+1}t_{j+2}\cdots t_\ell}. 
\end{align*} 
\item [\mylabel{coh}{\normalfont{(2)}}] For each integer $i\leq 2(n-d_\ell+1)-2$ there is an isomorphism
\[
\HH^i (\Hilb^{d_1+1,\ldots,d_\ell+1}(\BA^n),\BZ) \cong \HH^i (\Flag(d_1,\ldots,d_\ell,n),\BZ).
\]
In particular, there is an isomorphism  of graded rings
\[
\HH^\ast (\Hilb^{d_1+1,\ldots,d_\ell+1}(\BA^\infty),\BZ) \cong \BZ\left[c^{(j)}_1,\ldots,c^{(j)}_{d_{j+1}-d_j}\,\big|\,0\leq j\leq \ell-1\right]
\]
where $\deg c_k^{(j)} = 2k$ for all $k$.
\end{itemize}
\end{thm}

The case $\ell=1$ of \Cref{thm:flag-intro}\ref{coh} can also be proved as a consequence of the work \cite{Hilb^infinity}. Our proof does not involve any $\BA^1$-homotopy techniques as in \cite{Hilb^infinity}.

\begin{conventions*}
We work over the field $\BC$ of complex numbers. By a \emph{stratification} of a scheme $Y$ we mean a finite collection of locally closed subschemes $Z_1,\ldots,Z_p \subset Y$ such that the induced morphism $\coprod_{1\leq i\leq p} Z_i \to Y$ is bijective.
\end{conventions*}

\subsection*{Acknowledgments} We thank Camilla Felisetti and Emanuele Pavia for useful discussions. S.M. is particularly grateful to Kamil Rychlewicz and Dimitri Wyss for many discussions on the stabilisation of mixed Hodge structures.

S.M. is supported by  the HORIZON-MSCA-2024-PF-01 Project 
101203281 ``Moduli Spaces of Sheaves: Geometry and Invariants'', funded by the Research and Innovation framework programme Horizon Europe {\normalsize\euflag}. M.G., S.M., R.M., A.R. are members of GNSAGA of INDAM, Italy. A.A. was supported by Dan Petersen's Wallenberg Scholar fellowship.


\section{Quot schemes and Hilbert--Samuel strata}

\subsection{Infinite Quot schemes and Grassmannians}
Let $n,r,d\geq 1$ be integers.
Let $X$ be an $n$-dimensional smooth quasiprojective variety, $\CE$ a locally free  $\OO_X$-module of rank $r$. The \emph{Quot scheme of points} $\Quot^d(\CE) $ is the fine moduli space parametrising isomorphism classes of quotients $[\CE\onto T]$, where $T$ is a 0-dimensional coherent sheaf on $X$ such that $\chi(T)=d$. Two quotients are identified when their kernels coincide.

Given a closed point $p\in X$, one defines the \emph{punctual Quot scheme} to be the closed subscheme 
\[ 
\begin{tikzcd}
\Quot^d (\CE)_p\arrow[hook]{r} &  \Quot^d (\CE)
\end{tikzcd}
\]
consisting of quotients $[\CE\onto T]$ such that $T$ is entirely supported at $p$. This scheme is projective and only depends on $\widehat{\OO}_{X,p}$, see e.g.~\cite[Thm.~4.3]{Fantechi-Ricolfi-motivic} or \cite[\S 2.1]{ricolfi2019motive}.

Let now $(X, p)=(\BA^n, 0)$. Any closed embedding $\BA^n\hookrightarrow \BA^{n+1}$ naturally induces compatible closed embeddings
\[
\begin{tikzcd}
\Quot^d(\OO^{\oplus r}_{\BA^n})\arrow[hook]{r} & \Quot^d(\OO^{\oplus r}_{\BA^{n+1}}),\\   
\Quot^d(\OO^{\oplus r}_{\BA^n})_0 \arrow[hook]{r}\arrow[hook]{u} & \Quot^d(\OO^{\oplus r}_{\BA^{n+1}})_0\arrow[hook]{u}
\end{tikzcd}
\]
of the corresponding moduli spaces. We define the \emph{(punctual) Quot scheme of the infinite affine space} to be the colimit of the (punctual) Quot schemes
\begin{align*}
      \Quot^d(\OO^{\oplus r}_{\BA^\infty}) &\defeq \displaystyle\varinjlim   \Quot^d(\OO_{\BA^n}^{\oplus r}),\\
      \Quot^d(\OO^{\oplus r}_{\BA^\infty})_0 &\defeq\displaystyle\varinjlim   \Quot^d(\OO^{\oplus r}_{\BA^n})_0
\end{align*}
as ind-schemes. By the compatibility at finite levels, there is a canonical closed immersion of ind-schemes
\[
\begin{tikzcd}
\Quot^d(\OO^{\oplus r}_{\BA^\infty})_0
\arrow[hook]{r} & 
\Quot^d(\OO^{\oplus r}_{\BA^\infty}).
\end{tikzcd}
\]
 
Analogously, let $\Gr(d,n)$ be the Grassmaniann of $d$-dimensional linear subspaces of $\BC^n$. The direct limit along the natural closed immersions $\Gr(d,n) \into \Gr(d,n+1)$ defines the \emph{infinite Grassmannian} 
\[
\Gr(d,\infty) \defeq \varinjlim \Gr(d,n).
\]

\subsection{Hilbert--Samuel stratification}
We recall now the notion of  Hilbert--Samuel function,  a classical invariant attached to    graded modules, see e.g~\cite[\S5.1]{EISENBUD} for a more complete exposition.

Consider the polynomial ring $R=\BC[x_1,\ldots,x_n]$ with the standard grading $\deg(x_k)=1$, for $k=1,\ldots,n$. Throughout, we denote by $\mathfrak m\subset R$ the maximal ideal of the origin $0 \in \BA^n$.  
Let $M=\bigoplus_{i\in\BZ} M_i$ be a finite graded $R$-module. The \emph{Hilbert--Samuel function} $h_M$ attached to $M$ is the function $i \mapsto h_M(i) =\dim_{\BC} M_i$, for every $i\in \BZ$. 

\begin{definition}
Let $z=[\OO_{\BA^n}^{\oplus r}\onto T]\in  \Quot^d(\OO^{\oplus r}_{\BA^n})_0 $ be a closed point corresponding to a quotient supported at the origin $0\in \BA^n$. Set $\mathfrak m_T=\mathfrak m\otimes_R T$. We associate to $z$ the finite graded $R$-module 
\begin{align*}
    \mathsf{gr}(z) =\bigoplus_{i\geq 0} \mathfrak m_T^{i}/\mathfrak m_T^{i+1}.
\end{align*}
We define the \emph{Hilbert--Samuel function } of $z$ to be $h_z=h_{ \mathsf{gr}(z)}$.
\end{definition}
 
We have the following properties:
\begin{enumerate}
    \item The Hilbert--Samuel function $h_z$ has finite support, that is, it vanishes for almost all values. We will thus represent $h_z=(h_z(0),h_z(1), \ldots, h_z(t))$ as a finite string of integers, where $t$ is the largest index such that $h_z(t)\neq 0$. Such integer $t$ will be denoted by $\ell(h_z)$ and called the \emph{length} of the Hilbert--Samuel function $h_z$.
    \item We have that  $h_z(i)<r\cdot h_{R}(i)$. Note that $h_z(0)$ coincides with the minimal number of generators of the $R$-module $T$.
    \item The length $\chi(T)=d$ can be computed as the finite sum 
    $\lvert h_z\rvert = \sum_{i\geq 0}h_z(i)$.
\end{enumerate} 

\begin{definition}
 Given a function $h\colon\BZ\rightarrow \BN$ with finite support, and  integers $n,r\geq 1$, the \emph{Hilbert--Samuel stratum} $H_{h}^{n,r}$  is the (possibly empty) locally closed subset
\[
H_{h}^{n,r}=\Set{z\in \Quot^{\lvert h\rvert}(\OO^{\oplus r}_{\BA^n})_{0} | h_z=h}\subset \Quot^{\lvert h \rvert}(\OO^{\oplus r}_{\BA^n})_{0}.
\]
We endow $H_{h}^{n,r}$ with the reduced induced subscheme structure. When $n$ and $r$ are clear from context, we simply write $H_h$ to simplify the notation.
\end{definition}

The \emph{Hilbert--Samuel stratification} of the punctual Quot scheme  $\Quot^{d}(\OO^{\oplus r}_{\BA^n})_{0}$ is
\begin{equation}\label{HSstratification}
    \Quot^{d}(\OO^{\oplus r}_{\BA^n})_{0}=\coprod_{\substack{\lvert h\rvert=d}}H^{n,r}_{h},
\end{equation}
where the coproduct runs over all the Hilbert--Samuel functions $h$ such that $\lvert h\rvert = d$.

\subsection{Generalised Białynicki-Birula decomposition}
The Quot scheme $\Quot^{d}(\OO^{\oplus r}_{\BA^n})$ is almost never smooth, but we can still apply to it the modern version of the  Białynicki-Birula decomposition \cite{JJ_dec}, which generalises the classical theory of \cite{BBdecomposition}. We follow the presentation of \cite[Sec.~5]{Joachim-Sivic}, see also \cite[Sec.~3]{ELEMENTARY} for the case of Hilbert schemes.

Set $Q \defeq \Quot^{d}(\OO^{\oplus r}_{\BA^n})$ to ease notation. The diagonal $\BG_{\mathrm{m}}$-action on $\BA^n$ naturally lifts to a $\BG_{\mathrm{m}}$-action on the Quot scheme $Q$ and on the punctual Quot scheme $Q_0 \defeq \Quot^{d}(\OO^{\oplus r}_{\BA^n})_0$. By \cite{JJ_dec}, there exists a scheme $Q^+$, along  with morphisms
\[
\begin{tikzcd}[row sep=large]
Q^+\arrow[r, "\theta"]\arrow[d, "\rho"] & Q=\Quot^{d}(\OO^{\oplus r}_{\BA^n})\\
    Q^{\BG_{\mathrm{m}}},&
\end{tikzcd}
\]
such that $\theta$ induces a bijection between the closed points of $Q^+$ and those of the closed subscheme $Q_0 \subset Q$, as proved in \cite[Lemma 5.3]{Joachim-Sivic}.
The $\BG_{\mathrm{m}}$-fixed locus $Q^{\BG_{\mathrm{m}}}$ naturally decomposes as
\begin{equation}
\label{decomposition-homogeneous}
Q^{\BG_{\mathrm{m}}}=\coprod_{|h|=d}\CH_h,
\end{equation}
where the disjoint union is over all possible Hilbert--Samuel functions of size $\lvert h\rvert=d$, and the  `homogeneous locus' $\CH_h$ parametrises \emph{graded quotients} $\BC[x_1,\ldots,x_n]^{\oplus r} \onto T$ of $\BC[x_1,\ldots,x_n]$-modules (such that $T$ has Hilbert--Samuel function $h$). Pulling back the decomposition \eqref{decomposition-homogeneous} along $\rho$, we obtain a stratification 
\begin{align*}
Q^+=\coprod_{\lvert h\rvert=d} Q^+_{h},
\end{align*}
and the morphisms $\theta, \rho$ induce morphisms on  each  component of $ Q^+$
\begin{equation}
\label{eqn:rho-h}
\begin{tikzcd}
Q^{+}_{h,\mathsf{red}}\arrow[dr, "\theta_h"]\arrow[d, hook] &\\
   Q^+_{h}\arrow[d, "\rho_h"] & H_{h}\\
   \CH_h&
\end{tikzcd}
\end{equation}
where, once again, $\theta_h$ is a bijection on closed points, and the first vertical arrow is the closed immersion of the reduction. The morphism  $\rho_h$ should be thought as  a \emph{retraction} of the generalised  Białynicki-Birula scheme $Q_{h}^+$   onto its $\BG_{\mathrm{m}}$-fixed locus $  \CH_h$.

\subsection{Linear locus}
We introduce a distinguished locus of the punctual Quot scheme, which we call the \emph{linear locus}, and will play a key role in the proof of \Cref{thm:main-intro-A}.
\begin{definition}
\label{def:linearquot}
Let $n,r,d\geq 1$ be integers. The \emph{linear locus} in $ \Quot^d(\OO^{\oplus r}_{\BA^n})_0$ is the  closed subset
   \[
    \LQuot^d(\OO^{\oplus r}_{\BA^n})=\Set{z\in \Quot^d(\OO^{\oplus r}_{\BA^n})_0| \ell(h_z) < 2} .
   \]
\end{definition}

By construction, the decomposition in \eqref{HSstratification} restricts to the linear locus, thus inducing the locally closed stratification
\begin{equation}
\label{eqn:strata-linear-quot}
 \LQuot^d(\OO^{\oplus r}_{\BA^n})=\coprod_{ \substack{\lvert h\rvert=d \\ \ell(h)< 2}}H^{n,r}_{h}.
\end{equation}
We show now that the strata building up $\LQuot^d(\OO^{\oplus r}_{\BA^n})$ are all smooth, being suitably described in terms of Grassmannian fibrations.

\begin{prop}
\label{prop: linear locus}
Let $n,r,d \geq 1$ be positive integers and set $Q=\Quot^{d}(\OO^{\oplus r}_{\BA^n})$. Let $h$ be a Hilbert--Samuel function of length $\ell(h) < 2$ and size $\lvert h \rvert = d$. Then
\begin{enumerate}
    \item [\mylabel{homogeneous-locus}{\normalfont{(1)}}] The homogeneous locus $\CH_h\subset Q^{\BG_{\mathrm{m}}}$ is a Grassmann bundle over $\Gr(h(0),r)$ with fibre equal to $\Gr(h(1),h(0)\cdot n)$. In particular, it is a smooth projective variety.
    \item [\mylabel{strati}{\normalfont{(2)}}] The map $\rho_h \colon Q^+_{h} \to \CH_h$ obtained in \eqref{eqn:rho-h} is a vector bundle of rank \[h(1)(r-h(0)).\] In particular,
    $H_h\cong Q^{+}_{h}$ is a smooth quasiprojective variety.
\end{enumerate}
\end{prop}

\begin{proof} 
Let 
\[ 
\begin{tikzcd}
\OO^{\oplus r}_{\Gr(h(0),r)}\arrow[two heads]{r} & \CT_0
\end{tikzcd}
\]
be  the universal quotient on the Grassmannian $\Gr(h(0),r)$, and consider the  relative Grassmannian
\[
\begin{tikzcd}
\Gr\left(h(1),\CT_0^{\oplus n}\right)\arrow{r} & \Gr(h(0),r),
\end{tikzcd}
\]
whose fibres over a closed point $[\BC^r\onto T_0]\in \Gr(h(0),r)$ are given by 
$ \Gr(h(1),T_0^{\oplus n})$. We claim that 
\begin{align*}
    \CH_h\cong \Gr\left(h(1),\CT_0^{\oplus n}\right).
\end{align*}
To a closed point
\[
\begin{tikzcd}
\bigl({[}\BC^r\arrow[two heads]{r} & T_0{]}, {[} T_0^{\oplus n}\arrow[two heads]{r} & T_1{]} \bigr)\,\in\, \Gr\left(h(1),\CT_0^{\oplus n}\right)
\end{tikzcd}
\]
we associate the graded quotient
\[
\begin{tikzcd}
{\bigl[}
\OO_{\BA^n}^{\oplus r}\arrow[two heads]{r} & \BC^r\oplus \BC^{rn}\arrow[two heads]{r} & T_0\oplus T_0^{\oplus n}\arrow[two heads]{r} & T_0\oplus T_1
{\bigr]} \,\in\, \CH_h,
\end{tikzcd}
\]
where the first quotient is the projection onto the degree $\leq 1$ part of the $\BC$-graded vector space $\OO^{\oplus r}_{\BA^n}$. Conversely, suppose given a homogeneous quotient 
\[
\begin{tikzcd}
{[}\OO_{\BA^n}^{\oplus r}\arrow[two heads]{r} & T_0\oplus T_1{]}\,\in\, \CH_h.
\end{tikzcd}
\]
Taking the degree 0 and 1 parts yields quotients 
\[
\begin{tikzcd}
{\bigl[}\BC^r\arrow[two heads]{r} & T_0{\bigr]}, \qquad {\bigl[}\BC^{rn}\arrow[two heads]{r} & T_1{\bigr]}.
\end{tikzcd}
\]
Since $T_0$ generates $T_0\oplus T_1$ as an $\OO_{\BA^n}$-module, the second quotient factors as
\[
\begin{tikzcd}
 \BC^{rn}\arrow[two heads]{r} & T_0^{\oplus n}\arrow[two heads]{r} & T_1,
\end{tikzcd}
\]
and thus we can construct the closed point 
\[
\begin{tikzcd}
\bigl({[}\BC^r\arrow[two heads]{r} & T_0{]}, {[} T_0^{\oplus n}\arrow[two heads]{r} & T_1{]} \bigr)\,\in\, \Gr\left(h(1),\CT_0^{\oplus n}\right).
\end{tikzcd}
\]
The two associations are readily checked to give a bijection at the level of closed points  and to extend to flat families, yielding the desired isomorphism.

The proof of \ref{strati} works along the same lines as the proofs of \cite[Lemma~3.2]{Hilb_11} and \cite[Prop.~4.3]{8POINTS}.  
Consider the universal short exact sequence over  $Q^{\BG_{\mathrm{m}}}\times \BA^n$
\[
\begin{tikzcd}
0\arrow{r} & \CJ\arrow{r} & \OO^{\oplus r}_{Q^{\BG_{\mathrm{m}}}\times \BA^n}\arrow{r} & \CK\arrow{r} & 0,
\end{tikzcd}
\]
and let $\pi\colon Q^{\BG_{\mathrm{m}}}\times \BA^n\to  Q^{\BG_{\mathrm{m}}}$ be the projection. The sheaves $\CJ$ and $\CK$ are naturally $\BZ$-graded, the sheaf\footnote{One can check that in this case $\CF\cong \lHom(\pi_*\CJ_0,\pi_*\CK_1)$.} 
\[
\CF=\pi_\ast \lHom(\CJ_0,\CK_1)
\]
is locally free, and $\rho \colon Q^+ \to Q^{\BG_{\mathrm{m}}}$ coincides with the structure map
\[
\begin{tikzcd}
Q^+\cong \mathrm{Tot}_{Q^{\BG_{\mathrm{m}}}}(\CF) = \Spec \Sym \CF^\vee \arrow{r} & Q^{\BG_{\mathrm{m}}}.
\end{tikzcd}
\]
Now fix a Hilbert--Samuel function as in the statement, and a point 
\[
\begin{tikzcd}
z = 
{[}J\arrow[hook]{r} &  \OO^{\oplus r}\arrow[two heads]{r} & K{]} \in \CH_h\subset  Q^{\BG_{\mathrm{m}}}.
\end{tikzcd}
\]
The rank of $\rho_h = \rho|_{\CH_h}$ is computed as 
\begin{align*}
    \rk \CF\big|_z &=\dim_{\BC} \Hom_{\OO_{\BA^n}}(J_0, K_1)\\
    &=(r-h(0))h(1).
\end{align*}
In particular, the Białynicki-Birula stratum $Q_{h}^+$ is smooth and therefore isomorphic to $H_h$.
\end{proof}

We next show that  the Hilbert--Samuel strata $H_h$ appearing in the complement of the linear locus are locally in bijection with affine fibrations.

\begin{prop}
\label{prop: non linear pieces}
Let $n,r \geq 1$ be integers and let $h$ be a Hilbert--Samuel function of length $t=\ell(h)\geq 2$. Then each point $[K]\in H^{n,r}_h$ is contained in an open  $ U_K\subset H_h^{n,r}$ such that there exists a scheme $ W_K$ and a bijective morphism 
\[
\begin{tikzcd}[row sep =tiny]
   W_K\times \BA^{(r-h(0)+h(0)n-h(1))h(t)  } \arrow[r] & U_K.
\end{tikzcd}
\] 
\end{prop}
{ \begin{proof} 
Consider a closed point $[K]\in H_h^{n,r}$.
 Then, by  \cite{tandardBasis},  the existence of standard basis with respect to the local ordering given by $x_1<\dots < x_n <1$,  
    we can write  
    \begin{equation}\label{eq:standbas}
         K=\langle A_K \rangle_R+\langle B_K \rangle_R+\langle C_K \rangle_R,
    \end{equation}
    where $A_K, B_K, C_K\subset K$ are such that the initial degree (with respect to the standard grading) of all the elements  in $A_K,B_K,C_K$ is respectively $0, 1 $ and larger than $1$. 
   By the assumption on $h$, we have
   \begin{align*}
    |A_K|&=r-h(0),\\
    |B_K|&=h(0)n-h(1).
   \end{align*} 
 For every standard basis as in \eqref{eq:standbas}, define $\mdeg_{B_K}= \max_{b\in B_{K}} \deg b$ and similarly for $A_K$, and 
   \[
    \mdeg K=\max\{ \min_{\eqref{eq:standbas}} \mdeg_{B_K}, \min_{\eqref{eq:standbas}} \mdeg_{A_K}\}.
    \]
    Fix  subspaces
    \[
   \Psi_K\subset R_0^{\oplus r}, \quad \Delta_K \subset R_1^{\oplus r},\quad \Gamma_K \subset R_{t}^{\oplus r}
    \]
such that the restrictions of the projection $R^{\oplus r} \onto R^{\oplus r}/K$ to $\Psi_K,\Delta_K$ and $ \Gamma_K$  are  injective. Consider the open neighbourhood $[K]\in U_K\subset H_h^{n,r} $ consisting of the points $[K']\in H_h^{n,r}$ tuch that the restrictions of the projection $R^{\oplus r} \onto R^{\oplus r}/K'$ to $\Psi_K,\Delta_K$ and $ \Gamma_K$  are still  injective.
     
Now, let $W_K\subset U_K$ be the locus
\[
W_K=\Set{[K']\in U_K | \mdeg K'< t },
\]
endowed with the reduced structure. The initial elements of $A_K$ (resp.~$B_K$) span a subspace 
\begin{align*}
\Lambda_K&=\langle \In a \,|\,a \in A_K\rangle_{\BC}= (\In K)_0\subset R^{\oplus r}_0,\\
\Theta_K&=\langle \In b \,|\,b \in B_K\rangle_{\BC}\subset (\In K)_1\subset R^{\oplus r}_1.
\end{align*}
Note that $\Theta_K\cong\BC^{ h(0)n-h(1) }$ and $\Lambda_K\cong \BC^{ r-h(0)}$ only depend on $K$ and not on $A_K, B_K$. Set 
\[
\Sigma_K=\langle x_i\cdot\In a  |\  a \in A_K,\ i=1,\ldots,n \rangle_{\BC},
\]
and consider the direct sum decomposition and the projection
\begin{equation}\label{eq:decoR1}
    \begin{tikzcd}[row sep = tiny]
R_0^{\oplus r}\oplus R_1^{\oplus r}\cong  &[-1.1cm] \Psi_K\oplus\Lambda_K\oplus\Sigma_K\oplus \Delta_K \oplus \Theta_K\arrow[r,"\pi'"] &\Lambda_K\oplus  \Theta_K\\
&(a,u,b,c,v)\arrow[r,mapsto]& (u,v).
\end{tikzcd}
\end{equation} 
Denote by $\pi\colon R^{\oplus r} \to \Lambda_K\oplus \Theta_K$ the graded $\BC$-linear map obtained by precomposing $\pi'$ with the projection $R^{\oplus r}\onto R_0^{\oplus r}\oplus  R_1^{\oplus r}$. 
Let $[K']\in W_K$ and choose a standard basis as in \eqref{eq:standbas}. For every $\mu \in\Hom_\BC(\Lambda_K\oplus \Theta_K, \Gamma_K)$  consider the map 
\begin{equation}\label{eq:biho}
    \begin{tikzcd}[row sep =tiny]
   A_{K'}\cup B_{K'} \arrow[r, "\psi_{\mu}"] & R^{\oplus r} \\
    b\arrow[r,mapsto] &b+ \mu (\pi (b)) .
\end{tikzcd}
\end{equation}
Note that, if $     K'=\langle A_{K'} \rangle_R+\langle B_{K'} \rangle_R+\langle C_{K'} \rangle_R$ is a standard basis, then 
\[
    \langle \psi_{\mu}A_{K'} \rangle_R+\langle \psi_\mu B_{K'} \rangle_R+\langle C_{K'} \rangle_R
\]
is the standard basis of a (unique) module corresponding to a closed point of $ U_K$. 
This induces  the required bijective  morphism 
\[
\begin{tikzcd}[row sep =tiny]
   W_K\times \BA^{(r-h(0)+h(0)n-h(1))h(t) } \arrow[r] & U_K  \\  
   {([\langle A_{K'}\rangle_R+\langle B_{K'}\rangle_R+\langle C_{K'}\rangle_R], \mu)}\arrow[r,mapsto] & {[\langle \psi_\mu A_{K'}\rangle_R+\langle \psi_{\mu}B_{K'}\rangle_R+\langle C_{K'}\rangle_R]},
\end{tikzcd}
\]
where we identified $ \BA^{(r-h(0)+h(0)n-h(1))h(t)}\cong \Spec\,\Sym^\bullet(\Hom_\BC(\Lambda_K\oplus \Theta_K, \Gamma_K)^\vee)$.
 \end{proof} 
}

\section{Motivic stabilisation}
\subsection{Grothendieck ring of varieties}
 The \emph{Grothendieck ring of varieties}, denoted $K_0(\Var_{\BC})$, is the free abelian group generated by isomorphism classes  $[X]$ of 
 finite type $\BC$-varieties, modulo the relations   $[X] = [Z] + [{X\setminus Z}]$ whenever $Z \into X$ is a closed subvariety of $X$.  The operation $ [X]\cdot [Y] = [X\times Y]$
defines   a ring structure on $K_0(\Var_{\BC})$.

The Grothendieck ring of varieties comes equipped with a \emph{power structure},  that is a map
\[
\begin{tikzcd}[row sep=tiny]
   (1+tK_0(\Var_{\BC})\llbracket t \rrbracket) \times K_0(\Var_{\BC}) \arrow{r} & 1+tK_0(\Var_{\BC})\llbracket t \rrbracket\\
   (A(t),m) \arrow[mapsto]{r} & A(t)^m,
\end{tikzcd}
\]
satisfying natural compatibilities, see e.g.~\cite{GLMps}. If $Y$ is a smooth quasiprojective variety of dimension $n$ and $\CE$ is a locally free $\OO_Y$-module of rank $r$, one has the relation
\[
\sum_{d\geq 0}\,[\Quot^d(\CE)]t^d = \left(\sum_{d\geq 0}\,[\Quot^d(\OO_{\BA^n}^{\oplus r})_0]t^d \right)^{[Y]},
\]
which says that the motives of the punctual Quot schemes determine the motives of the full Quot schemes, regardless of the bundle $\CE$ chosen \cite{ricolfi2019motive}. In this respect, analysing the structure of the punctual motives completely solves the `motivic Quot theory'.

\smallbreak
In what follows, we  also work with the complete topological ring
\[
\widehat{K_0}(\Var_\BC) = \varprojlim_{n} K_0(\Var_\BC) / \BL^n K_0(\Var_\BC),
\]
namely the $\BL$-adic completion of $K_0(\Var_\BC)$, where $\BL \defeq [\BA^1]\in {K_0}(\Var_\BC)$ denotes the \emph{Lefschetz motive}.

\subsection{Proof of the main theorem}
We start with some preliminary results.

\begin{lemma}
\label{lemma:infinite-grass-motive}
Fix an integer $d\geq 0$. There is an identity
    \begin{align*}
[\Gr(d,\infty)]=\prod_{k=1}^d\frac{1}{1-\BL^k}\in \widehat{K_0}(\Var_\BC).
    \end{align*}
\end{lemma}

\begin{proof}
We proceed by induction, the base step $[\Gr(0,\infty)]=1$ being clear.
By the standard $\BL$-binomial identity (see e.g.~\cite[Lemma 4.1]{MOTIVEMOTIVEMOTIVE}) we have
    \[
    [\Gr(d+1,n+1)] = [\Gr(d,n)]+\BL^{d+1}[\Gr(d+1,n)]. 
    \]
Taking the limit for $n\to \infty$ yields the relation
\begin{align*}
   [\Gr(d+1,\infty)] = [\Gr(d,\infty)]+\BL^{d+1}[\Gr(d+1,\infty)],
\end{align*}
which implies, by the inductive hypothesis, that 
\[
[\Gr(d+1,\infty)]=\frac{1}{1-\BL^{d+1}} [\Gr(d,\infty)]=\prod_{k=1}^{d+1}\frac{1}{1-\BL^{k}}.\qedhere
\]
\end{proof}
See \Cref{eqn:infinite-flag-motive} for a generalisation of \Cref{lemma:infinite-grass-motive} to the infinite flag variety.

\begin{remark}
The motive of the infinite Grassmannian can  be alternatively  extracted from the closed formula \cite[Eqn.~(2.6)]{GLMstacks}
\begin{align*}\label{eqn: inft grass formula}
    \sum_{d=0}^\infty \,[\Gr(d,\infty)]t^d=\prod_{i=0}^\infty\frac{1}{1-\BL^{i}t}\in \widehat{K_0}(\Var_\BC)\llbracket t \rrbracket.
\end{align*}
\end{remark}

\begin{prop}\label{lemma: prelim}
Fix integers $r,d \geq 1$. Then, there is an identity in $\BZ[\BL] \subset K_0(\Var_{\BC})$
\[
\sum_{i=1}^d\Bigg(\prod_{k=d-i+1}^{d-1}(1-\BL^k)\Bigg)[\Gr(i,r)]\BL^{(d-i)(r-i)} = \sum_{k=0}^{r-1} \BL^{dk}.
\]
\end{prop}
\begin{proof}
    Let us put $\mathsf P_i(d) \defeq \prod_{k=d-i+1}^{d-1}(1-\BL^k)$, for $i\geq 1$, with the convention that $\mathsf P_1(d)=1$. The left-hand-side of the formula becomes
    \[\mathsf R(r,d)\defeq \sum_{i=1}^d \mathsf P_i(d)[\Gr(i,r)]\BL^{(d-i)(r-i)}.\]
The proof is concluded provided we show that $\mathsf R(r,d)$ satisfy the same recurrence relations of $\sum_{k=0}^{r-1} \BL^{dk}$, namely, for every $d\geq 1$, that $\mathsf R(1,d)=1$  and $\mathsf R(r+1,d)-\mathsf R(r,d)=\BL^{dr}$. 

The initial condition $\mathsf R(1,d)=1$ follows immediately from the definition. By the $\BL$-binomial identity 
\[[\Gr(i,r+1)] = [\Gr(i,r)]+\BL^{r+1-i}[\Gr(i-1,r)], \]
we have
\[
\mathsf R(r+1,d) = \sum_{i=1}^d \mathsf P_i(d)[\Gr(i,r)] \BL^{(d-i)(r+1-i)}  + \sum_{i=1}^d \mathsf P_i(d)[\Gr(i-1,r)]\BL^{(d-i)(r+1-i)+(r+1-i)}.
\]
Therefore
\begin{multline}
\label{eqn:madonna-psicotica}
    \mathsf R(r+1,d)-\mathsf R(r,d) = \sum_{i=1}^d \mathsf P_i(d)[\Gr(i,r)]
\left(\BL^{(d-i)(r+1-i)} - \BL^{(d-i)(r-i)}\right) \\
+ \sum_{i=1}^d \mathsf P_i(d)[\Gr(i-1,r)]\BL^{(d-i+1)(r-i+1)}.
\end{multline}
Rearranging the sums, the right hand side of \Cref{eqn:madonna-psicotica} becomes 
\[ \sum_{i=1}^d  \mathsf P_i(d)[\Gr(i,r)]
\left(\BL^{d-i}-1\right)\BL^{(d-i)(r-i)} \\
+ \sum_{j=0}^{d-1} \mathsf P_{j+1}(d)[\Gr(j,r)]\BL^{(d-j)(r-j)}.
\]
Notice that the summand for $i=d$ in the first sum vanishes. Thus, combining the two sums back yields 
\[\mathsf R(r+1,d)-\mathsf R(r,d) = \BL^{dr} + \sum_{i=1}^{d-1} \big(\mathsf P_i(d)(\BL^{d-i}-1) + \mathsf P_{i+1}(d)\big) [\Gr(i,r)] \BL^{(d-i)(r-i)}.\]
Finally, we have
\[
\mathsf P_{i+1}(d)=  \prod_{k=d-i}^{d-1}(1-\BL^k) =(1-\BL^{d-i})\mathsf P_i(d),
\] 
by which we conclude that $\mathsf R(r+1,d)-\mathsf R(r,d) = \BL^{dr}$, as required.
\end{proof}

We are ready to prove \Cref{thm:main-intro-A} from the Introduction.

\begin{theorem}
\label{main-A}
Fix integers $d,r\geq 1$. There is an identity
\[
[\Quot^d(\OO_{\BA^\infty}^{\oplus r})_0]=[\Gr(d-1, \infty)]\cdot \sum_{i=0}^{r-1} \BL^{di}\in \widehat{K_0}(\Var_\BC).
\]
\end{theorem}

\begin{proof}
Let $h$ be a Hilbert--Samuel function of length $t\geq 2$. 
By \Cref{prop: non linear pieces} and a standard stratification argument, we have that $\BL^{(r-h(0)+h(0)n-h(1)) \cdot h(t)}$ divides the motive $[H_h^{n,r}]$, which implies that the difference 
\[
[\Quot^d(\OO_{\BA^n}^{\oplus r})_0]-[\LQuot^d(\OO_{\BA^n}^{\oplus r})]\in \BL^{n+r-d+1}K_0(\Var_{\BC})
\]
is divisible by a suitable power of $\BL$. In particular, this implies the  convergence
\begin{align}\label{eqn: convergence to 0}
    \lim_{n\to \infty}\left(  [\Quot^d(\OO^{\oplus r}_{\BA^n})_0]-[\LQuot^d(\OO_{\BA^n}^{\oplus r})]\right)=0\in \widehat{K_0}(\Var_\BC).
\end{align}
We claim that the limit $ \lim_{n\to \infty} [\LQuot^d(\OO_{\BA^n}^{\oplus r})]$ exists and is given by
\begin{align}\label{eqn: claim linear}
    \lim_{n\to \infty}[\LQuot^d(\OO_{\BA^n}^{\oplus r})]=[\Gr(d-1, \infty)]\cdot \sum_{i=0}^{r-1} \BL^{di}.
\end{align}
In particular, combining \eqref{eqn: convergence to 0} with \eqref{eqn: claim linear} concludes the proof.

To prove the claim, let $h $ be a Hilbert--Samuel function of length $\ell(h)<2$. Then, by \Cref{prop: linear locus}, we have 
\begin{equation}
\label{eqn: motives leq 2}
  [H^{n,r}_h]=[\Gr(h(0),r)][\Gr(h(1),h(0)\cdot n)]\BL^{h(1)(r-h(0))}.
\end{equation}
The possible Hilbert--Samuel functions $h$ of length $\ell(h)<2$ are of the form 
\[(1,d-1), (2,d-2),  \dots, (d,0),\] 
therefore combining \eqref{eqn: motives leq 2} and \eqref{eqn:strata-linear-quot} with one another we obtain
\begin{equation}
\label{eqn: motive linear quot}
[\LQuot^d(\OO_{\BA^n}^{\oplus r})]=\sum_{i=1}^d\,[\Gr(i,r)][\Gr(d-i,in)]\BL^{(d-i)(r-i)}.
\end{equation}
The existence of the limit $   \lim_{n\to \infty}[\LQuot^d(\OO^{\oplus r}_{\BA^n})]$ easily follows by the existence of the limit
\begin{align*}
    \lim_{n\to \infty}[\Gr(d-i,in)]=[\Gr(d-i,\infty)]\in  \widehat{K_0}(\Var_\BC),
\end{align*}
for all $d\geq i \geq 1$. This implies that
\begin{align}\label{eqn: limit on motives}
\begin{split}
     \lim_{n\to \infty}[\LQuot^d(\OO^{\oplus r}_{\BA^n})]&=   \sum_{i=1}^d\,[\Gr(i,r)][\Gr(d-i,\infty)]\BL^{(d-i)(r-i)} \\
 &= [\Gr(d-1,\infty)]\sum_{i=1}^d[\Gr(i,r)]\left(\prod_{k=d-i+1}^{d-1}(1-\BL^k)\right)\BL^{(d-i)(r-i)}\\
 &= [\Gr(d-1,\infty)]\cdot \sum_{i=0}^{r-1} \BL^{di},
 \end{split}
\end{align}
where in the second equality we used \Cref{lemma:infinite-grass-motive} and in the third equality we used \Cref{lemma: prelim}. This proves the claim and concludes the proof.
\end{proof}

\begin{remark}
Via a direct check, as done in \cite[\S 9.1]{MOTIVEMOTIVEMOTIVE}, one can show that the approximation order 
\[
[\Quot^d(\OO_{\BA^n}^{\oplus r})_0]-[\LQuot^d(\OO_{\BA^\infty}^{\oplus r})]\in \BL^{n+r-d+1}K_0(\Var_{\BC})
\]
in the proof of \Cref{main-A} is sharp. Indeed, the    Hilbert--Samuel stratum   $H^{2,1}_{(1,1,1)}$ is isomorphic to the total space of a line bundle on $\BP^1$. Hence, its motive is divisible by $ \BL$ but not by  $\BL^2$.
\end{remark}

\section{Cohomological stabilisation}
\label{sec:cohomology}

Let $Y$ be a complex algebraic variety. Its \emph{Hodge--Deligne polynomial} is
\[
\mathsf E(Y;u,v) = \sum_{p,q,i}(-1)^i h^{p,q}(\HH^i_c(Y,\BQ))u^pv^q.
\]
Sending $[Y] \mapsto E(Y;u,v)$ is a motivic measure $K_0(\Var_{\BC}) \to \BZ[u,v]$, see \cite{Danilov-Khovanskij}.
The \emph{weight polynomial}, namely the specialisation
\[
\mathsf w(Y,z) = \mathsf E(Y;z,z) = \sum_{p,q,i}(-1)^i h^{p,q}(\HH^i_c(Y,\BQ))z^{p+q} \,\in\,\BZ[z],
\]
coincides with the \emph{signed Poincar\'e polynomial}
\[
P(Y,-z) = \sum_{i \geq 0}\dim_{\BQ} \HH^i(Y,\BQ)(-z)^i
\]
when $Y$ is smooth and projective, by the Hodge decomposition theorem. We define the \emph{Poincar\'e series} $P(Y,z)$ of an arbitrary topological space $Y$ to be the (possibly infinite) sum $\sum_{i \geq 0}\dim_{\BQ} \HH^i(Y,\BQ)z^i$. This will be used in \Cref{cor:poincare}.
 
\begin{theorem}
\label{thm:purity}
Fix integers $d,r\geq 1$. Deligne's mixed Hodge structure on the cohomology $\HH^k(\LQuot^d(\OO_{\BA^n}^{\oplus r}), \BQ)$ is pure of Tate type for all $k$. The integral cohomology $\HH^k(\LQuot^d(\OO_{\BA^n}^{\oplus r}), \BZ)$ is torsion-free. In particular, there is an identity
\begin{equation}
\label{eqn:poincarepoly}
P( \LQuot^d(\OO_{\BA^n}^{\oplus r}), z)=\sum_{i=1}^dP(\Gr(i,r), z)P(\Gr(d-i,in), z)z^{2(d-i)(r-i)}.
\end{equation}
\end{theorem}
\begin{proof}
Set $L=\LQuot^d(\OO_{\BA^n}^{\oplus r})$. The conditions of being pure of Tate type and torsion-freeness follow at once from the stronger assertion that $L$ has an affine paving, i.e.~an increasing filtration of closed subschemes
\[
\emptyset = X_{-1} \subset X_0 \subset X_1 \subset \cdots \subset X_\ell = L
\]
such that $X_i \setminus X_{i-1}$ is an affine space. We start from a filtration
\[
\emptyset = L_{-1} \subset L_0 \subset L_1 \subset \cdots \subset L_{d-1} = L
\]
where $L_i \subset L$ is defined as
\[
L_i = \Set{z \in L | h_z(0) \geq d-i} \subset L.
\]
But $h_z(0)$ coincides with the minimal number of generators of the quotient module determined by $z$. Given a flat family $\CF$ of 0-dimensional sheaves (supported at the origin) on a base scheme $S$, the function $S \to \BN$ taking $s$ to the minimal number of generators of $\CF_s/\Fm\cdot \CF_s$ is upper semi-continuous: this says that the loci $L_i$ are closed. We endow them with the reduced subscheme structure. We have locally closed stratifications for each $L_i$, given by
\begin{align*}
    L_{0} &= H_{(d,0)} = \Gr(d,r) \\
    L_1 &= H_{(d,0)} \amalg H_{(d-1,1)} \\
    \vdots & \\
    L_{d-1} &= H_{(d,0)} \amalg H_{(d-1,1)} \amalg \cdots \amalg H_{(1,d-1)} =L.
\end{align*}
We are reduced to exhibiting an affine paving for each $H_{(d-i,i)}$. Indeed, if 
\[
\emptyset = Y_{-1}^{(i)} \subset Y_{0}^{(i)} \subset \cdots \subset Y_{s_i-1}^{(i)} \subset Y_{s_i}^{(i)} = H_{(d-i,i)}
\]
is an affine paving, in the $i$th step we have
\[
\cdots \subset L_i \subset L_i \amalg Y_0^{(i+1)} \subset L_i \amalg Y_1^{(i+1)} \subset \cdots \subset L_i \amalg Y_{s_{i+1}-1}^{(i+1)} \subset L_i \amalg Y_{s_{i+1}}^{(i+1)} = L_{i+1}\subset \cdots
\]

To confirm that each $H_{(d-i,i)}$ has an affine paving, we exploit \Cref{prop: linear locus}. There, we proved that $H_{(d-i,i)}$ is a vector bundle over a Grassmann bundle $\CH_{(d-i,i)} \to \Gr(d-i,r)$. It is straightforward to check that any such structure admits an affine paving, see \cite{Ful}.
 
Finally, to prove \eqref{eqn:poincarepoly}, we observe that by purity the weight polynomial of $\LQuot^d(\OO_{\BA^n}^{\oplus r})$, which by \eqref{eqn: motive linear quot} coincides with the right hand side of the equation, agrees with the Poincar\'e polynomial. 
\end{proof}

The next result identifies the cohomology of the full Quot scheme with the cohomology of the punctual Quot scheme, and will be used in the proof of \Cref{cor:poincare} below.

\begin{lemma}
\label{prop: totaro}
Let $d, n, r\geq 1$. Then  the natural inclusion 
    \[
    \begin{tikzcd}
    \Quot^d(\OO^{\oplus r}_{\BA^n})_0\arrow[hook]{r} & \Quot^d(\OO_{\BA^n}^{\oplus r})
    \end{tikzcd}
    \]
    is a homotopy equivalence in the classical topology. In particular, there is a graded ring isomorphism
\begin{align*}
\HH^*\left(\Quot^d(\OO^{\oplus r}_{\BA^n}), \BZ\right) \cong     \HH^*\left(\Quot^d(\OO^{\oplus r}_{\BA^n})_0, \BZ\right).
    \end{align*}
\end{lemma}
\begin{proof}
The proof is a straightforward generalisation of  \cite[Cor.~2.3]{Totaro-Hilb-homotopic}, which follows as a corollary of  \cite[Thm.~2.2]{Totaro-Hilb-homotopic}.
\end{proof}

\begin{theorem}
\label{cor:poincare}
Fix integers $d,r, n\geq 1$.    The restriction  
\[
\begin{tikzcd}
    \HH^*\left(\Quot^d(\OO^{\oplus r}_{\BA^n}), \BZ\right)\arrow[r] &\HH^*\left(\LQuot^d(\OO^{\oplus r}_{\BA^n}), \BZ\right)
\end{tikzcd}
\]
is an isomorphism in degrees at most $2\left( n+r-d\right)$. In particular, 
there is a graded isomorphism of cohomology rings 
\[
\HH^*\left(\Quot^d(\OO^{\oplus r}_{\BA^\infty}), \BZ\right)\cong \HH^*\left(\LQuot^d(\OO^{\oplus r}_{\BA^\infty}), \BZ\right),
\]
and an identity
\[
P(\Quot^d(\OO^{\oplus r}_{\BA^\infty}), z) = \left(\prod_{k=1}^{d-1}\frac{1}{1-z^{2k}}\right)\cdot \sum_{i=0}^{r-1}z^{2di}.
\]
\end{theorem}
\begin{proof}
Consider the closed embedding 
\[ 
\begin{tikzcd}
\LQuot^d(\OO_{\BA^n}^{\oplus r})\arrow[hook]{r} &   \Quot^d(\OO_{\BA^n}^{\oplus r})_0
\end{tikzcd}
\]
    of the linear locus in the punctual Quot scheme, and set its complement to be  \[U_n= \Quot^d(\OO_{\BA^n}^{\oplus r})_0\setminus  \LQuot^d(\OO_{\BA^n}^{\oplus r}).\]
    The long exact sequence in compactly supported cohomology yields 
\[
\begin{tikzcd}
0\arrow[r]
    &  \HH_c^0(U_n, \BZ) \arrow[r]
    & \HH^0\left( \Quot^d(\OO_{\BA^n}^{\oplus r})_0, \BZ\right) \arrow[r]
        \arrow[d, phantom, ""{coordinate, name=Z}]
      & \HH^0\left( \LQuot^d(\OO_{\BA^n}^{\oplus r}), \BZ\right) \arrow[dll,rounded corners,
to path={ -- ([xshift=2ex]\tikztostart.east)
|- (Z) [near end]\tikztonodes
-| ([xshift=-2ex]\tikztotarget.west) -- (\tikztotarget)}] \\
& \HH_c^1(U_n, \BZ) \arrow[r]
    & \HH^1\left( \Quot^d(\OO_{\BA^n}^{\oplus r})_0, \BZ\right) \arrow[r]
& \HH^1\left( \LQuot^d(\OO_{\BA^n}^{\oplus r}), \BZ\right) \arrow[r] & \cdots
\end{tikzcd}
\]
where we used that $ \Quot^d(\OO_{\BA^n}^{\oplus r})_0$ and $  \LQuot^d(\OO_{\BA^n}^{\oplus r})$ are compact in the classical topology. By \Cref{prop: non linear pieces}, we have the vanishing
    \begin{align*}
        \HH^{i}_c(U_n, \BZ)=0, \quad \mbox{for } \, i<2( n+r-d+1),
    \end{align*}
which by the long exact sequence above yields isomorphisms
\begin{align*}
\HH^{i}\left( \Quot^d(\OO_{\BA^n}^{\oplus r})_0, \BZ\right) \cong \HH^i\left( \LQuot^d(\OO_{\BA^n}^{\oplus r}), \BZ\right), \quad \mbox{for } \, i\leq 2\left( n+r-d+1\right)-2,
\end{align*}
which in turn induce a graded ring isomorphism on the colimit
    \begin{align*}
          \HH^{*}\left( \Quot^d(\OO_{\BA^\infty}^{\oplus r})_0, \BZ\right) \cong \HH^*\left( \LQuot^d(\OO_{\BA^\infty}^{\oplus r}), \BZ\right).
    \end{align*}
Combined with \Cref{prop: totaro}, this settles the first claim.
    
    Finally, the Poincaré series  $P(\LQuot^d(\OO^{\oplus r}_{\BA^\infty}), z) $ is computed analogously to \eqref{eqn: limit on motives} by replacing $\BL$  with $z^2$. 
\end{proof}

\section{Proof of \texorpdfstring{\Cref{panda-conjecture}}{Theorem~\ref{panda-conjecture}}}
\label{sec:main-proof}

In this section, we prove \Cref{panda-conjecture}, answering in the positive a question by Pandharipande.

\begin{theorem}
\label{main-thm-pandha}
Fix integers $d,r \geq 1$. Then, there is a graded ring isomorphism
\[
\begin{tikzcd}
\BZ[c_1,\ldots,c_{d}]/(c_d^r) \arrow{r}{\sim} & \HH^\ast(\Quot^d(\OO^{\oplus r}_{\BA^\infty}), \BZ),
\end{tikzcd}
\]
where $\deg c_i=2i$ for each $i=1, \ldots, d$.
\end{theorem}

\begin{proof}
Let $n \geq 1$ be an integer, and let $\CE$ be the vector bundle of rank $d$ on $\Quot^d(\OO^{\oplus r}_{\BA^n})$ whose fibre over a closed point $z=[\OO^{\oplus r}_{\BA^n} \onto T]$ is the vector space of global sections 
\[\CE|_z=\HH^0(T).      \] Let $c \in \HH^0(\OO_{\BA^n})$ be a non-zero constant section. Then, for any surjection $\pi \colon \OO^{\oplus r}_{\BA^n} \onto T$ we have that $\pi(c,\dots,c) \in \HH^0(T^{\oplus r})$ is a non-vanishing section. It follows that the top Chern class of $\CE^{\oplus r}$ vanishes:
\[
c_{rd}(\CE^{\oplus r}) = c_d(\CE)^r = 0.
\]
Further passing to the colimit $n \to \infty$, we may thus define a graded ring homomorphism
\[
\begin{tikzcd}[row sep=tiny]
\BZ[c_1,\ldots,c_{d}]/(c_d^r) \arrow{r}{\phi} & \HH^\ast(\Quot^d(\OO^{\oplus r}_{\BA^\infty}), \BZ) \\
c_i \arrow[mapsto]{r} & c_i(\CE).
\end{tikzcd}
\]
\Cref{thm:main-B} shows that the domain and codomain of $\phi$ are free abelian groups and have the same finite rank in each cohomological degree. Hence, to conclude that it is an isomorphism it suffices to show that it has determinant $\pm 1$ degree-wise. It is then enough to see that $\phi$ is injective after reducing modulo any prime $p$. As we lack control over the cohomology ring $\HH^\ast(\Quot^d(\OO^{\oplus r}_{\BA^\infty}), \BZ)$, we shall show the injectivity of $\phi$ after post-composition with another morphism $\psi^*$ which we introduce below.

Let $\CH_{(1,d-1)}^{n,r} \subset H^{n,r}_{(1,d-1)} \subset \Quot^d(\OO^{\oplus r}_{\BA^n})$ be the homogeneous locus with Hilbert--Samuel function $h = (1,d-1)$. We have seen in the proof of \cref{prop: linear locus} that this locus is isomorphic to the relative Grassmannian
\[
    \CH_{(1,d-1)}^{n,r} \cong \Gr\left(d-1,\CT_0^{\oplus n}\right),
\]
where $\CT_0$ is the universal quotient bundle on the Grassmannian $\Gr(1,r)$. Under the identification $\Gr(1,r) \cong \BP^{r-1}$, we have $\CT_0 \cong \OO_{\BP^{r-1}}(1)$. Let $p_1,p_2$ be the projections
\[
\begin{tikzcd}
\BP^{r-1} & \BP^{r-1} \times \Gr(d-1,n) \arrow{l}[swap]{p_1} \arrow{r}{p_2} & \Gr(d-1,n)
\end{tikzcd}
\]
and let $\OO^{\oplus n}_{\Gr(d-1,n)} \onto \CQ$ be the tautological quotient bundle. By the universal property of the relative Grassmannian, the surjection
\[
\begin{tikzcd}
p_1^*\OO_{\BP^{r-1}}(1)^{\oplus n} \cong p_1^*\OO_{\BP^{r-1}}(1) \otimes p_2^*\OO_{\Gr(d-1,n)}^{\oplus n} \arrow[two heads]{r} & p_1^*\OO_{\BP^{r-1}}(1) \otimes p_2^*\CQ
\end{tikzcd}
\]
corresponds to a morphism
\[
\begin{tikzcd}
\BP^{r-1} \times \Gr(d-1,n) \arrow{r}{f} & \Gr\left(d-1,\OO_{\BP^{r-1}}(1)^{\oplus n}\right)
\end{tikzcd}
\]
over $\BP^{r-1}$. In concrete terms, it is given on closed points by
\[
\begin{tikzcd}
(L \subset \BC^r, P \subset \BC^n) \arrow[mapsto]{r} & (L \subset \BC^r, L^\vee \otimes P \subset L^\vee \end{tikzcd}
\otimes \BC^n)
\]
where $P$ is a $(d-1)$-plane, $L$ is a line, and we have identified the fibre of $\OO_{\BP^{r-1}}(1)$ over $L$ with the dual vector space $L^\vee$. Notice that, by definition, we have an isomorphism
\begin{equation}
\label{eqn:pullback-universal-property}
f^*\CT_1 \cong p_1^*\OO_{\BP^{r-1}}(1) \otimes p_2^*\CQ
\end{equation}
where $\CT_1$ denotes the tautological quotient bundle on the relative Grassmannian. Observe as well that the explicit isomorphism given in the proof of \cref{prop: linear locus} shows that the restriction of the bundle $\CE$ to the homogeneous locus $\CH_{(1,d-1)}^{n,r}$ is
\begin{equation}
\label{eqn:restriction-tautological-homogeneous}
\CE\big|_{\CH_{(1,d-1)}^{n,r}} \cong q^*\OO_{\BP^{r-1}}(1) \oplus \CT_1
\end{equation}
where $q \colon \Gr\left(d-1,\OO_{\BP^{r-1}}(1)^{\oplus n}\right) \to \BP^{r-1}$ is the projection.

Letting $n \to \infty$, we write $\psi$ for the composite map
\[
\begin{tikzcd}
\psi \colon \BP^{r-1} \times \Gr(d-1,\infty) \arrow{r}{f^\infty} &  \Gr\left(d-1,\OO_{\BP^{r-1}}(1)^{\oplus \infty}\right) \cong \CH_{(1,d-1)}^{\infty,r} \arrow[hook]{r} & \Quot^d(\OO^{\oplus r}_{\BA^\infty}).
\end{tikzcd}
\]
The cohomology ring of the domain of $\psi$ is 
\[
\HH^\ast(\BP^{r-1} \times \Gr(d-1,\infty)) \cong \BZ[h,e_1,\dots,e_{d-1}]/(h^r)
\] with $h$ in degree $2$ and $e_i$ in degree $2i$. By \eqref{eqn:pullback-universal-property} and \eqref{eqn:restriction-tautological-homogeneous}, we have an isomorphism
\[
    \psi^*\CE \cong p_1^*\OO_{\BP^{r-1}}(1) \otimes \left(\OO_{\BP^{r-1} \times \Gr(d-1,n)} \oplus p_2^*\CQ \right).
\]
To finish the proof, we need to confirm that the composition
\[
\begin{tikzcd}[row sep=tiny]
\BZ[c_1,\ldots,c_{d}]/(c_d^r) \arrow{r}{\phi} & \HH^\ast(\Quot^d(\OO^{\oplus r}_{\BA^\infty}), \BZ)\arrow{r}{\psi^\ast} & \BZ[h,e_1,\dots,e_{d-1}]/(h^r)
\end{tikzcd}
\]
is injective modulo any prime.
A computation of Chern classes shows that
\[
    \psi^*(c_i(\CE)) \equiv e_i \mod h, \quad \text{if } i < d 
\]
and
\[
    \psi^*(c_d(\CE)) = h e_{d-1} + h^2 e_{d-2} + \dots + h^d.
\]

Finally, let $p$ be a prime number and write a generic non-zero element of $\BZ/p[c_1,\dots,c_d]/(c_d^r)$ as a polynomial in $c_d$ with coefficients $P_i$ in $\BZ/p[c_1,\dots,c_{d-1}]$, i.e.
\[
    0 \neq x = \sum_{i = 0}^{r-1} c_d^i P_i(c_1,\dots,c_{d-1}).
\]
Let $j$ be the smallest index such that $P_j \neq 0$. Then, using the Chern class computation above, looking at the smallest power of $h$ gives: 
\[
    \psi^* \circ \phi(x) \equiv h^j e_{d-1}^j P_j(e_1,\dots,e_{d-1}) \mod h^{j+1}.
\]
This is non-zero in $\BZ/p[h,e_1,\dots,e_{d-1}]/(h^r)$. Therefore $\psi^* \circ \phi$ is injective modulo any prime, hence so is $\phi$, which finishes the proof.
\end{proof}

We   give moreover a topological interpretation of the stable cohomology appearing in \Cref{main-thm-pandha}.

\begin{corollary}\label{cor: sphere}
Fix an integer $r \geq 1$. Let $\gamma_d^{\oplus r} \setminus \{0\}$ denote the complement of the zero section of the $r$-fold sum of the tautological rank $d$ vector bundle $\gamma_d$ on the infinite Grassmannian $\Gr(d,\infty)$. Then there is a continuous map (in the analytic topology)
\[
\begin{tikzcd}
\Quot^d(\OO^{\oplus r}_{\BA^\infty}) \arrow{r} &  \gamma_d^{\oplus r} \setminus \{0\}
\end{tikzcd}
\]
which induces the isomorphism of \cref{main-thm-pandha} on integral cohomology.
\end{corollary}

\begin{proof}
As in the proof of \cref{main-thm-pandha}, let $n \geq 1$ be an integer and let $\CE$ be the vector bundle on $\Quot^d(\OO^{\oplus r}_{\BA^n})$ whose fibre over a closed point $[\OO^{\oplus r}_{\BA^n} \onto T]$ is the vector space of global sections $\HH^0(T)$. As a vector bundle of rank $d$, it is classified by the homotopy class of a map
\begin{equation}
\label{eqn:classifying-map}
\begin{tikzcd}
\Quot^d(\OO^{\oplus r}_{\BA^n}) \arrow{r} & \Gr(d,\infty).
\end{tikzcd}
\end{equation}
We shall first show that this map lifts along the bundle $\gamma_d^{\oplus r} \setminus \{0\} \to \Gr(d,\infty)$. By construction, $\gamma_d \setminus \{0\}$ is homotopy equivalent to the homotopy orbits $(\BC^d \setminus \{0\}) \sslash U(d)$, hence
\[
\gamma_d^{\oplus r} \setminus \{0\} \simeq (\Hom(\BC^r,\BC^d) \setminus 0) \sslash U(d).
\]
Let $\mathrm{Fr}(\CE)$ be the frame bundle, whose fibre over a point $[\OO^{\oplus r}_{\BA^n} \onto T]$ is given by the space of invertible linear maps $\mathrm{Isom}(\BC^d, \HH^0(T))$. We may then define the morphism
\[
\begin{tikzcd}[row sep=tiny]
\mathrm{Fr}(\CE) \arrow{r} & \Hom(\BC^r,\BC^d) \setminus \{0\} \\
\left([\pi \colon \OO^{\oplus r}_{\BA^n} \onto T], \varphi \colon \BC^d \overset{\sim}{\to} \HH^0(T) \right) \arrow[mapsto]{r} & \varphi^{-1} \circ \HH^0(\pi)_0
\end{tikzcd}
\]
where $\HH^0(\pi)_0$ is the restriction of $\HH^0(\pi) \colon \HH^0(\OO^{\oplus r}_{\BA^n}) \to \HH^0(T)$ to the degree $0$ part $\BC^r \subset \HH^0(\OO^{\oplus r}_{\BA^n})$. This morphism is $U(d)$-equivariant, so we may pass to the homotopy orbits to obtain a continuous map
\[
\begin{tikzcd}
\Quot^d(\OO^{\oplus r}_{\BA^n}) \simeq \mathrm{Fr}(\CE) \sslash U(d) \arrow{r} & (\Hom(\BC^r,\BC^d) \setminus \{0\}) \sslash U(d) \simeq \gamma_d^{\oplus r} \setminus \{0\}
\end{tikzcd}
\]
above $* \sslash U(d) \simeq \Gr(d,\infty)$. Passing to the colimit $n \to \infty$, we obtain that the map~\eqref{eqn:classifying-map} factors as
\[
\begin{tikzcd}
\Quot^d(\OO^{\oplus r}_{\BA^\infty}) \arrow{r} & \gamma_d^{\oplus r} \setminus \{0\} \arrow{r} & \Gr(d,\infty).
\end{tikzcd}
\]
\Cref{main-thm-pandha} shows that this composite induces the quotient morphism
\[
\begin{tikzcd}
\BZ[c_1,\dots,c_d] \arrow{r} & \BZ[c_1,\dotsc,c_d]/(c_d^r)
\end{tikzcd}
\]
on cohomology. But so does the second map, as the Gysin sequence of the bundle $\gamma_d^{\oplus r} \setminus \{0\}$ shows. Hence the first map induces an isomorphism on cohomology.
\end{proof}

\begin{remark}
In the case $r = 1$, we see that $\gamma_d \setminus \{0\}$ is homotopy equivalent to the Grassmannian $\Gr(d-1,\infty)$ and the map
\[
\begin{tikzcd}
\Hilb^d(\BA^\infty) \arrow{r} & \Gr(d-1,\infty)
\end{tikzcd}
\]
was shown to be a homology equivalence (and even a weak homotopy equivalence) in \cite[Cor.~4.9]{Hilb^infinity}. 
\end{remark}

\section{The nested Hilbert scheme case}
\label{sec:nested-hilb}
In this section we show how our techniques apply to the nested setting, in the case $r=1$. Here the main character for the motivic stabilisation of the punctual nested Hilbert scheme 
\[
\Hilb^{d_1+1,\ldots,d_\ell+1}(\BA^n)_0 = \Set{(Z_1,\ldots,Z_\ell)|Z_1\subset \cdots\subset Z_\ell,\,\Supp(Z_\ell) = 0}\subset \prod_{i=1}^\ell \Hilb^{d_i+1}(\BA^n)_0
\]
is the \emph{flag variety} $\Flag(d_1,\ldots,d_\ell,n)$, the smooth projective variety parametrising flags $V_1\subset \cdots \subset V_\ell$ of subspaces $V_i \subset \BC^n$ of increasing dimensions $d_i = \dim_{\BC} V_i$. A straightforward argument (see e.g.~\cite[Sec.~6.4]{monavari2024hyperquot}) allows one to compute the motive
\begin{equation}
\label{eqn:finite-flag-motive}
[\Flag(d_1,\ldots,d_\ell,n)] = \frac{\prod_{k=1}^{n}\,\left(\BL^k-1\right)}{\prod_{j=0}^{\ell}\prod_{k=1}^{d_{j+1}-d_{j}}\,\left(\BL^k-1\right)}\,\in\, K_0(\Var_\BC).
\end{equation}
Here, as customary, we have set $d_{\ell+1}=n$ and $d_0 = 0$.
Let us compute the limit
\[
[\Flag(d_1,\ldots,d_\ell,\infty)] = \lim_{n\to \infty}\,[\Flag(d_1,\ldots,d_\ell,n)]\,\in\,\widehat{K}_0(\Var_{\BC}).
\]
This will generalise the calculation in \Cref{lemma:infinite-grass-motive}. We first simplify terms in \Cref{eqn:finite-flag-motive} and isolate the dependence on $n$ by rewriting
\[
[\Flag(d_1,\ldots,d_\ell,n)] = \prod_{k=n-d_\ell+1}^n\,(\BL^k-1)\cdot \frac{1}{\prod_{j=0}^{\ell-1}\prod_{k=1}^{d_{j+1}-d_{j}}\,\left(\BL^k-1\right)}.
\]
In the limit, the first factor becomes simply $(-1)^{d_\ell}$. The denominator can be rewritten as
\[
\prod_{j=0}^{\ell-1}\prod_{k=1}^{d_{j+1}-d_{j}}\,\left(\BL^k-1\right) = (-1)^{\sum_{j=0}^{\ell-1}(d_{j+1}-d_{j})} \prod_{j=0}^{\ell-1}\,\prod_{k=1}^{d_{j+1}-d_{j}}\,\left(1-\BL^k\right).
\]
Adding the sign $(-1)^{d_\ell}$ back in gives a total contribution of $(-1)^{2d_\ell} = 1$, therefore
\begin{equation}
\label{eqn:infinite-flag-motive}
[\Flag(d_1,\ldots,d_\ell,\infty)] = \prod_{j=0}^{\ell-1}\,\prod_{k=1}^{d_{j+1}-d_{j}}\,\frac{1}{1-\BL^k} = \prod_{j=0}^{\ell-1}\,[\Gr(d_{j+1}-d_{j},\infty)].
\end{equation}

Our final goal is to prove the following nested analogue of the rank 1 version of \Cref{main-A}.

\begin{theorem}\label{thm:nested}
Fix integers $\ell\geq 1$ and $0\leq d_1\leq \cdots \leq d_\ell \leq n$. Then 
\begin{equation}
\label{nested-bounded}[\Hilb^{d_1+1,\ldots,d_\ell+1}(\BA^n)_0] =[\Flag(d_1,\ldots, d_\ell,n)] \in K_0(\Var_\BC)/\BL^{n-d_\ell+1}K_0(\Var_\BC).
\end{equation}
In particular, the  sequence of motives $([\Hilb^{d_1+1,\ldots,d_\ell+1}(\BA^n)_0])_n$ converges in $ \widehat{K}_0(\Var_\BC)$ and in $\widehat{K}_0(\Var_\BC)\llbracket t \rrbracket$ we have the identity
\begin{align*}
\sum_{0\leq {d_1\leq\cdots\leq d_\ell}}\,[\Hilb^{d_1+1,\ldots,d_\ell+1}(\BA^\infty)_0]t_1^{d_1}\cdots t_\ell^{d_\ell} &=\prod_{j=0}^{\ell-1}\prod_{i\geq 0}\frac{1}{1-\BL^it_{j+1}t_{j+2}\cdots t_\ell}. 
\end{align*}
\end{theorem}

The next definition relies on the standard fact that the theory of Hilbert--Samuel functions naturally extends to the nested setup, see e.g.~\cite{UPDATES,Graffeo-Lella-new-components} and \cite[Sec.~2.4]{GL:2steps}.

\begin{definition}
Consider the vector of Hilbert--Samuel functions $\boldit{h}=((1,d_i))_{i=1}^\ell$. In analogy with \Cref{def:linearquot} we define the \emph{linear locus} in $ \Hilb^{d_1+1,\ldots,d_\ell+1}(\BA^n)_0$ to be the closed subset
   \begin{equation}
   \label{eqn:nested-linear}
\LHilb^{d_1+1,\ldots,d_\ell+1}({\BA^n})=H_{\boldit{h}} \subset \Hilb^{d_1+1,\ldots,d_\ell+1}(\BA^n)_0.
   \end{equation}
\end{definition}
It is straightforward to observe that there is an isomorphism 
\begin{equation}
\label{iso-flag-linearlocus}
\LHilb^{d_1+1,\ldots,d_\ell+1}({\BA^n})\cong \Flag(d_1,\ldots,d_\ell,n).
\end{equation}

\begin{proofof}{\Cref{thm:nested}}
Let $\boldit{h}=(h_1\leq \cdots\leq h_\ell)$ be an $\ell$-tuple of Hilbert--Samuel functions such that $h_\ell$ has length $t\geq 2$ and $h_i(0)=1$, for $i=1,\ldots,\ell$. A similar argument to that of \Cref{prop: non linear pieces} shows that the motive of nonlinear locus in the nested Hilbert scheme, namely the complement of \eqref{eqn:nested-linear}, is divisible by $\BL^{(n-h_\ell(1))h_\ell(t)}$. Therefore, thanks to the isomorphism \eqref{iso-flag-linearlocus}, we obtain
\[
[\Hilb^{d_1+1,\ldots,d_\ell+1}(\BA^n)_0] -[\Flag(d_1,\ldots,d_\ell,n)] \,\in\, \BL^{n-d_\ell+1}K_0(\Var_{\BC}),
\]
which yields \eqref{nested-bounded} and in particular the key identity
\[
[\Hilb^{d_1+1,\ldots,d_\ell+1}(\BA^\infty)_0] = [\Flag(d_1,\ldots,d_\ell,\infty)]\,\in\,\widehat{K}_0(\Var_\BC).
\]
Set $m_j = d_{j+1}-d_j$ for $j = 0,\ldots,\ell-1$. Then $d_i = \sum_{0\leq j\leq i-1}m_j$ for all $i=1,\ldots,\ell$ (recall that $d_0=0$) and 
\[
t_1^{d_1}\cdots t_\ell^{d_\ell} = t_1^{m_0}t_2^{m_0+m_1}\cdots t_{\ell}^{m_0+\cdots+m_{\ell-1}} = \prod_{j=0}^{\ell-1}\,\left(t_{j+1} t_{j+2}\cdots t_\ell\right)^{m_j}.
\]
Therefore, setting $q_j = t_{j+1} t_{j+2}\cdots t_\ell$ and exploiting \Cref{eqn:infinite-flag-motive}, we can compute
\begin{align*}
\sum_{0\leq {d_1\leq\cdots\leq d_\ell}}\,[\Flag(d_1,\ldots,d_\ell,\infty)] t_1^{d_1}\cdots t_\ell^{d_\ell} 
&= \sum_{m_0,\ldots,m_{\ell-1}\geq 0} \prod_{j=0}^{\ell-1}\,\prod_{k=1}^{m_j}\,\frac{1}{1-\BL^k} q_j^{m_j} \\
&=\prod_{j=0}^{\ell-1}\sum_{m\geq 0}q_j^{m}\prod_{k=1}^{m}\,\frac{1}{1-\BL^k} \\
&=\prod_{j=0}^{\ell-1} \sum_{m\geq 0}[\Gr(m,\infty)]q_j^m \\
&=\prod_{j=0}^{\ell-1}\prod_{i\geq 0}\frac{1}{1-\BL^iq_j} \\
&=\prod_{j=0}^{\ell-1}\prod_{i\geq 0}\frac{1}{1-\BL^it_{j+1} t_{j+2}\cdots t_\ell}.\qedhere
\end{align*}
\end{proofof}

\begin{remark}
It is not immediate to formulate an analogue of \Cref{thm:nested} for nested Quot schemes, as there is no direct relation between the (nested) linear Quot scheme and a ``nice'' variety playing the role of the flag variety as in \eqref{iso-flag-linearlocus} for the case $r=1$. For instance, one can see that the morphism
\[
\begin{tikzcd}
\LQuot^{d_1+1,\ldots,d_\ell+1}(\OO_{\BA^n}^{\oplus r})\arrow[r]&\LQuot^{d_\ell+1}(\OO_{\BA^n}^{\oplus r})
\end{tikzcd}
\]
fails to be flat in general, unlike in the nested Hilbert scheme case.
\end{remark}

The following corollary generalises to the nested setup the isomorphism in cohomology obtained in \cite{Hilb^infinity} when $\ell=1$, without exploiting any  $\BA^1$-homotopy techniques.

\begin{corollary}
\label{cor:nested-cohomology}
Fix integers $\ell \geq 1$ and $0\leq d_1\leq \cdots d_\ell \leq n$. For each integer $i\leq 2(n-d_\ell+1)-2$ there is an isomorphism
\[
\HH^i(\Hilb^{d_1+1,\ldots,d_\ell+1}(\BA^n),\BZ) \cong \HH^i(\Flag(d_1,\ldots,d_\ell,n),\BZ).
\]
In particular, there is an isomorphism  of graded rings
\[
\HH^\ast (\Hilb^{d_1+1,\ldots,d_\ell+1}(\BA^\infty),\BZ) \cong \BZ\left[c^{(j)}_1,\ldots,c^{(j)}_{d_{j+1}-d_j}\,\big|\,0\leq j\leq \ell-1\right]
\]
where $\deg c_k^{(j)} = 2k$ for all $k$.
\end{corollary}

\begin{proof}
The first statement follows from the nested upgrade of \Cref{prop: totaro} and \Cref{cor:poincare} and using the isomorphism \eqref{iso-flag-linearlocus}. The second statement is a direct consequence of the first.
\end{proof}

\bibliographystyle{amsplain-nodash}
\bibliography{The_Bible}

\smallskip 
\noindent
{\small{Alexis Aumonier \\
\address{Department of Mathematics, Stockholms universitet, 106 91 Stockholm (Sweden)} \\
\href{mailto:alexis.aumonier@math.su.se}{\texttt{alexis.aumonier@math.su.se}}
}}

\smallskip 
\noindent
{\small{Michele Graffeo \\
\address{Dipartimento di Matematica, Politecnico di Milano, Piazza Leonardo da Vinci 32, 20133, Milan (Italy) \&  INFN, Sez. di Milano, Via Giovanni Celoria 16, 20133, Milano, Italy} \\
\href{mailto:michele.graffeo@polimi.it}{\texttt{michele.graffeo@polimi.it}}
}}

\smallskip
\noindent
{\small{Sergej Monavari \\
\address{Dipartimento di Matematica “Tullio Levi-Civita”, Università di Padova, Via Trieste 63, 35121 Padova (Italy)},
\href{mailto:sergej.monavari@math.unipd.it}{\texttt{sergej.monavari@math.unipd.it}}
}}

\smallskip
\noindent
{\small{Riccardo Moschetti,\\
\address{Università di Pavia, Via Ferrata 1, 27100, Pavia (Italy)},
\href{mailto:riccardo.moschetti@unipv.it}{\texttt{riccardo.moschetti@unipv.it}}
}}

\smallskip
\noindent
{\small Andrea T. Ricolfi \\
\address{SISSA, Via Bonomea 265, 34136, Trieste (Italy)},
\href{mailto:aricolfi@sissa.it}{\texttt{aricolfi@sissa.it}}}
\end{document}